\documentclass[12pt]{article}
\usepackage{amssymb}
\usepackage{amsthm}
\usepackage{amsmath}
\usepackage{graphicx}
\DeclareMathOperator{\Id}{Id}
\DeclareMathOperator{\BId}{\mathbf{Id}}
\DeclareMathOperator{\Con}{Con}
\DeclareMathOperator{\BCon}{\mathbf{Con}}
\DeclareMathOperator{\Ker}{Ker}
\DeclareMathOperator{\BKer}{\mathbf{Ker}}
\DeclareMathOperator{\T}{\Theta}
\DeclareMathOperator{\PO}{\Pi_1}
\DeclareMathOperator{\PT}{\Pi_2}
\newtheorem{theorem}{Theorem}[section]
\newtheorem{definition}[theorem]{Definition}

\newtheorem{proposition}[theorem]{Proposition}

\newtheorem{example}[theorem]{Example}
\newtheorem{corollary}[theorem]{Corollary}
\title{Directly decomposable ideals and congruence kernels of commutative semirings}
\author{Ivan Chajda, G\"unther Eigenthaler and Helmut L\"anger}
\date{}
\begin{document}
\footnotetext[1]{Support of the research by \"OAD, project CZ~02/2019, and support of the research of the first author by IGA, project P\v rF~2019~015, is gratefully acknowledged.}
\maketitle
\begin{abstract}
As pointed out in the monographs \cite G and \cite{KS} on semirings, ideals play an important role despite the fact that they need not be congruence kernels as in the case of rings. Hence, having two commutative semirings $\mathbf S_1$ and $\mathbf S_2$, one can ask whether an ideal $I$ of their direct product $\mathbf S=\mathbf S_1\times\mathbf S_2$ can be expressed in the form $I_1\times I_2$ where $I_j$ is an ideal of $\mathbf S_j$ for $j=1,2$. Of course, the converse is elementary, namely if $I_j$ is an ideal of $\mathbf S_j$ for $j=1,2$ then $I_1\times I_2$ is an ideal of $\mathbf S_1\times\mathbf S_2$. Having a congruence $\Theta$ on a commutative semiring $\mathbf S$, its $0$-class is an ideal of $\mathbf S$, but not every ideal is of this form. Hence, the lattice $\BId\mathbf S$ of all ideals of $\mathbf S$ and the lattice $\BKer\mathbf S$ of all congruence kernels (i.e.\ $0$-classes of congruences) of $\mathbf S$ need not be equal. Furthermore, we show that the mapping $\Theta\mapsto[0]\Theta$ need not be a homomorphism from $\BCon\mathbf S$ onto $\BKer\mathbf S$. Moreover, the question arises when a congruence kernel of the direct product $\mathbf S_1\times\mathbf S_2$ of two commutative semirings can be expressed as a direct product of the corresponding kernels on the factors. In the paper we present necessary and sufficient conditions for such direct decompositions both for ideals and for congruence kernels of commutative semirings. We also provide sufficient conditions for varieties of commutative semirings to have directly decomposable kernels.
\end{abstract}
 


\section{Introduction}

Semirings play an important role both in algebra and applications. There exist two different versions of this concept, namely semirings having a unit element (\cite G) and those without such an element (\cite{CL} and \cite{KS}). Since the second version is mostly used in applications, we define the basic concept of this paper as follows:

\begin{definition}
{\rm(}see {\rm\cite{KS})} A {\em commutative semiring} is an algebra $\mathbf S=(S,+,\cdot,0)$ of type $(2,2,0)$ satisfying
\begin{itemize}
\item $(S,+,0)$ is a commutative monoid,
\item $(S,\cdot)$ is a commutative semigroup,
\item $(x+y)z\approx xz+yz$,
\item $0x\approx0$.
\end{itemize}
\end{definition}

If $\mathbf S$ is a commutative semiring containing an element $1$ satisfying the identity $1x\approx x$ then $\mathbf S$ is called {\em unitary}, and if it satisfies the identity $xx\approx x$ then it is called {\em idempotent}. A (semi-)ring $(S,+,\cdot,0)$ is called {\em zero-{\rm(}semi-{\rm)}ring} if $xy=0$ for all $x,y\in S$.

It is evident that every (unitary) commutative ring is a (unitary) commutative semiring and that every distributive lattice $\mathbf L=(L,\vee,\wedge,0)$ with least element $0$ is an idempotent commutative semiring.

In the following let $\mathbb N$ denote the set of non-negative integers. Then clearly, $\mathbf N=(\mathbb N,+,\cdot,0)$ is a unitary commutative semiring. For every positive integer $a$ let $a\mathbb N$ denote the set $\{0,a,2a,3a,\ldots\}$ of all non-negative multiples of $a$. It is evident that $a\mathbf N=(a\mathbb N,+,\cdot,0)$ is again a commutative semiring which is unitary if and only if $a=1$.

We recall the following definition from \cite G:

\begin{definition}\label{def1}
Let $\mathbf S=(S,+,\cdot,0)$ be a commutative semiring. An {\em ideal} of $\mathbf S$ is a subset $I$ of $S$ satisfying
\begin{itemize}
\item $0\in I$,
\item if $a,b\in I$ then $a+b\in I$,
\item if $a\in I$ and $s\in S$ then $as\in I$.
\end{itemize}
\end{definition}

Note that in case that $\mathbf S$ is a ring, the ideals of $\mathbf S$ in the sense of Definition~\ref{def1} need not be ring ideals. Consider e.g.\ the zero-ring whose additive group is the group $(\mathbb Z,+,0)$ of the integers. Then the ideals of this zero-ring in the sense of Definition~\ref{def1} are the submonoids of $(\mathbb Z,+,0)$, whereas the ring ideals are the subgroups of $(\mathbb Z,+,0)$.

The converse is, of course, trivial for any ring $\mathbf S$: Every ring ideal of $\mathbf S$ is an ideal in the sense of Definition~\ref{def1}.

Let $\Id\mathbf S$ denote the set of all ideals of a commutative semiring $\mathbf S=(S,+,\cdot,0)$. It is clear that $\BId\mathbf S=(\Id\mathbf S,\subseteq)$ is a complete lattice with smallest element $\{0\}$ and greatest element $S$. Moreover, for $I,J\in\Id\mathbf S$ we have
\begin{align*}
  I\vee J & =I+J=\{i+j\mid i\in I,j\in J\}, \\
I\wedge J & =I\cap J.
\end{align*}
For $a\in S$ let $I(a)$ denote the ideal of $\mathbf S$ generated by $a$. Obviously, $I(a)=a\mathbb N+aS$. The lattice $\BId\mathbf S$ need not be modular as the following example shows:

\begin{example}\label{ex2}
The Hasse diagram of the ideal lattice of the commutative zero-semiring $\mathbf S=(S,+,\cdot,0)$ on $\{0,a,b,c,d,e,f,g\}$ defined by
\[
\begin{array}{c|cccccccc}
+ & 0 & a & b & c & d & e & f & g \\
\hline
0 & 0 & a & b & c & d & e & f & g \\
a & a & b & c & 0 & e & f & g & d \\
b & b & c & 0 & a & f & g & d & e \\
c & c & 0 & a & b & g & d & e & f \\
d & d & e & f & g & d & e & f & g \\
e & e & f & g & d & e & f & g & d \\
f & f & g & d & e & f & g & d & e \\
g & g & d & e & f & g & d & e & f
\end{array}
\]
looks as follows {\rm(}see Figure~1{\rm)}:
\vspace*{-2mm}
\begin{center}
\setlength{\unitlength}{8mm}
\begin{picture}(6,12)
\put(3,1){\circle*{.2}}
\put(5,3){\circle*{.2}}
\put(1,5){\circle*{.2}}
\put(5,5){\circle*{.2}}
\put(1,7){\circle*{.2}}
\put(3,7){\circle*{.2}}
\put(5,7){\circle*{.2}}
\put(3,9){\circle*{.2}}
\put(3,11){\circle*{.2}}
\put(3,1){\line(-1,2)2}
\put(3,1){\line(1,1)2}
\put(1,5){\line(0,1)2}
\put(5,3){\line(0,1)4}
\put(3,7){\line(-1,-1)2}
\put(3,7){\line(1,-1)2}
\put(3,9){\line(1,-1)2}
\put(3,11){\line(-1,-2)2}
\put(3,11){\line(0,-1)4}
\put(2.6,.35){$\{0\}$}
\put(-.45,4.85){$\{0,b\}$}
\put(5.25,2.85){$\{0,d\}$}
\put(5.25,4.85){$\{0,d,f\}$}
\put(-1.4,6.85){$\{0,a,b,c\}$}
\put(5.25,6.85){$\{0,d,e,f,g\}$}
\put(3.4,8.85){$\{0,b,d,e,f,g\}$}
\put(2.8,11.3){$S$}
\put(1.85,5.5){$\{0,b,d,f\}$}
\put(2.35,-.5){{\rm Fig.~1}}
\end{picture}
\end{center}
\vspace*{6mm}
and hence this lattice is not modular. Observe that $(S,+,0)$ is isomorphic to the direct product of its submonoids $\{0,a,b,c\}$ {\rm(}four-element cyclic group{\rm)} and $\{0,d\}$ {\rm(}two-element join-semilattice{\rm)}. The ideals of the semiring $\mathbf S$ are the submonoids of $(S,+,0)$.
\end{example}

Let $\BCon\mathbf S=(\Con\mathbf S,\subseteq)$ denote the congruence lattice of a commutative semiring $\mathbf S$. A {\em congruence kernel} of $\mathbf S$ is a set of the form $[0]\Theta$ with $\Theta\in\Con\mathbf S$. It is well known (cf.\ \cite G) that every congruence kernel is an ideal of $\mathbf S$, but not vice versa. Let
\[
\BKer\mathbf S=(\Ker\mathbf S,\subseteq)=(\{[0]\Theta\mid\Theta\in\Con\mathbf S\},\subseteq)
\]
denote the (complete) {\em lattice of congruence kernels} of $\mathbf S$. In contrast to rings, $\BCon\mathbf S$ and $\BKer\mathbf S$ need not be isomorphic as the following example shows, in which two different congruences have the same kernel.

\begin{example}
Consider the three-element lattice $\mathbf D_3=(D_3,\vee,\wedge,0)=(\{0,a,1\},\vee,\wedge,0)$. Then the Hasse diagram of $\BCon\mathbf D_3$ looks as follows {\rm(}see Figure~2{\rm)}:
\vspace*{-2mm}
\begin{center}
\setlength{\unitlength}{8mm}
\begin{picture}(6,6)
\put(3,1){\circle*{.2}}
\put(1,3){\circle*{.2}}
\put(5,3){\circle*{.2}}
\put(3,5){\circle*{.2}}
\put(3,1){\line(-1,1)2}
\put(3,1){\line(1,1)2}
\put(3,5){\line(-1,-1)2}
\put(3,5){\line(1,-1)2}
\put(2.8,.35){$\Delta$}
\put(.2,2.8){$\Theta_1$}
\put(5.25,2.8){$\Theta_2$}
\put(2.8,5.35){$\nabla$}
\put(2.35,-.5){{\rm Fig.~2}}
\end{picture}
\end{center}
\vspace*{3mm}
where
\[
\Theta_1=\{0,a\}^2\cup\{1\}^2\text{ and }\Theta_2=\{0\}^2\cup\{a,1\}^2.
\]
However, $\Delta$ and $\Theta_2$ have the same kernel $\{0\}$. Hence
\[
\BKer\mathbf D_3=(\{\{0\},\{0,a\},D_3\},\subseteq)
\]
is a three-element chain and $\BCon\mathbf D_3\not\cong\BKer\mathbf D_3$. Moreover, even the mapping $\Theta\mapsto[0]\Theta$ is not a homomorphism from $\BCon\mathbf D_3$ onto $\BKer\mathbf D_3$ since
\[
[0](\Theta_1\vee\Theta_2)=[0]\nabla=D_3\neq\{0,a\}=\{0,a\}\vee\{0\}=[0]\Theta_1\vee[0]\Theta_2.
\]
\end{example}

\section{Ideals of direct products of commutative semirings}

In the following we are interested in ideals on a direct product of two commutative semirings. Let $\mathbf S_1$ and $\mathbf S_2$ be commutative semirings. Of course, if $I_1\in\Id\mathbf S_1$ and $I_2\in\Id\mathbf S_2$ then $I_1\times I_2\in\Id(\mathbf S_1\times\mathbf S_2)$. An ideal $I$ of $\mathbf S_1\times\mathbf S_2$ is called {\em directly decomposable} if there exist $I_1\in\Id\mathbf S_1$ and $I_2\in\Id\mathbf S_2$ with $I_1\times I_2=I$. If $I$ is not directly decomposable then it is called a {\em skew ideal}. The aim of this paper is to characterize those commutative semirings which have directly decomposable ideals.

\begin{example}\label{ex1}
If $\mathbf R_2=(R_2,+,\cdot,0)=(\{0,1\},+,\cdot,0)$ denotes the two-element zero-ring and $\mathbf D_2=(D_2,\vee,\wedge,0)=(\{0,1\},\vee,\wedge,0)$ the two-element lattice then $\mathbf S:=\mathbf R_2\times\mathbf D_2$ has the non-trivial ideals
\begin{align*}
& \{0\}\times D_2, \\
& R_2\times\{0\}, \\
& \{(0,0),(0,1),(1,1)\}
\end{align*}
and hence $\BId\mathbf S\cong\mathbf N_5$ which is not modular and, moreover, the last mentioned ideal is skew.
\end{example}

\begin{example}\label{ex3}
If $\mathbf R_4=(R_4,+,\cdot,0)=(\{0,a,b,c\},+,\cdot,0)$ denotes the zero-ring whose additive group is the Kleinian $4$-group defined by
\[
\begin{array}{c|cccc}
+ & 0 & a & b & c \\
\hline
0 & 0 & a & b & c \\
a & a & 0 & c & b \\
b & b & c & 0 & a \\
c & c & b & a & 0
\end{array}
\]
then $\mathbf S:=\mathbf R_4\times\mathbf D_2$ has the non-trivial ideals
\begin{align*}
& \{0,a\}\times\{0\}, \\
& \{0,b\}\times\{0\}, \\
& \{0,c\}\times\{0\}, \\
& R_4\times\{0\}, \\
& \{0\}\times D_2, \\
& \{0,a\}\times D_2, \\
& \{0,b\}\times D_2, \\
& \{0,c\}\times D_2, \\
& \{(0,0),(0,1),(a,1)\}, \\
& \{(0,0),(0,1),(b,1)\}, \\
& \{(0,0),(0,1),(c,1)\}, \\
& \{(0,0),(0,1),(a,1),(b,1),(c,1)\}, \\
& \{(0,0),(a,0),(0,1),(a,1),(b,1),(c,1)\}, \\
& \{(0,0),(b,0),(0,1),(a,1),(b,1),(c,1)\}, \\
& \{(0,0),(c,0),(0,1),(a,1),(b,1),(c,1)\}
\end{align*}
the last seven of which are skew.
\end{example}

For sets $A$ and $B$ let $\pi_1$ and $\pi_2$ denote the first and second projection from $A\times B$ onto $A$ and $B$, respectively. Note that for any subset $C$ of $A\times B$ we have $C\subseteq\pi_1(C)\times\pi_2(C)$. Furthermore, if $C$ is of the form $A_1\times B_1$ with $A_1\subseteq A$ and $B_1\subseteq B$ then $\pi_1(C)=A_1$ and $\pi_2(C)=B_1$.

We borrow the method from \cite{FH} (which was used also in \cite{CEL18}) to prove the following theorem:

\begin{theorem}\label{th1}
Let $\mathbf S_1=(S_1,+,\cdot,0)$ and $\mathbf S_2=(S_2,+,\cdot,0)$ be commutative semirings and $I\in\Id(\mathbf S_1\times\mathbf S_2)$ and consider the following assertions:
\begin{enumerate}
\item[{\rm(i)}] $I$ is directly decomposable,
\item[{\rm(ii)}] $(S_1\times\{0\})\cap((\{0\}\times S_2)+I)\subseteq I$ and $((S_1\times\{0\})+I)\cap(\{0\}\times S_2)\subseteq I$,
\item[{\rm(iii)}] if $(a,b)\in I$ then $(a,0),(0,b)\in I$,
\item[{\rm(iv)}] $((S_1\times\{0\})+I)\cap((\{0\}\times S_2)+I)=I$.
\end{enumerate}
Then {\rm(iii)} $\Leftrightarrow$ {\rm(i)} $\Rightarrow$ {\rm(iv)} $\Rightarrow$ {\rm(ii)}.
\end{theorem}

\begin{proof}
$\mbox{}$ \\
(iii) $\Rightarrow$ (i): If $(a,b)\in\pi_1(I)\times\pi_2(I)$ then there exists some pair $(c,d)\in S_1\times S_2$ with $(a,d),(c,b)\in I$, hence $(a,0),(0,b)\in I$ which shows $(a,b)=(a,0)+(0,b)\in I$. \\
(i) $\Rightarrow$ (iii): This is clear. \\
(i) $\Rightarrow$ (iv): If $I=I_1\times I_2$ then
\begin{align*}
& ((S_1\times\{0\})+I)\cap((\{0\}\times S_2)+I)= \\
& =((S_1\times\{0\})+(I_1\times I_2))\cap((\{0\}\times S_2)+(I_1\times I_2))=(S_1\times I_2)\cap(I_1\times S_2)= \\
& =I_1\times I_2=I.
\end{align*}
(iv) $\Rightarrow$ (ii): This follows immediately.
\end{proof}

That (ii) does not imply (iii) can be seen by considering the ideal $I:=\{(0,0),(0,1),(a,1)\}$ of $\mathbf S$ in Example~\ref{ex3}. Since
\begin{align*}
(S_1\times\{0\})\cap((\{0\}\times S_2)+I) & =\{(0,0)\}\subseteq I, \\
((S_1\times\{0\})+I)\cap(\{0\}\times S_2) & =\{(0,0),(0,1)\}\subseteq I,
\end{align*}
(ii) holds. Because of $(a,1)\in I$ and $(a,0)\notin I$, (iii) does not hold. This shows that (ii) does not imply (iii). It is worth noticing that the implication (ii) $\Rightarrow$ (iii) holds in the case of commutative rings since in this case
\begin{align*}
(a,0) & =(0,-b)+(a,b), \\
(0,b) & =(-a,0)+(a,b).
\end{align*}
So in this case (i) and (ii) are equivalent.

\begin{example}
According to {\rm(iii)} of Theorem~\ref{th1}, the ideal $I(4,6)$ of $2\mathbf N\times2\mathbf N$ is not directly decomposable since
\begin{align*}
(4,0),(0,6)\notin I(4,6) & =(4,6)\mathbb N+(4,6)(2\mathbb N\times2\mathbb N)=(4,6)\mathbb N+(8\mathbb N\times12\mathbb N)= \\
                         & =(8\mathbb N\times12\mathbb N)\cup((4+8\mathbb N)\times(6+12\mathbb N)).
\end{align*}
\end{example}

Next we present several simple sufficient conditions for direct decomposability of ideals.

\begin{corollary}
Let $\mathbf S_1$ and $\mathbf S_2$ be commutative semirings such that one of the following conditions hold:
\begin{enumerate}
\item[{\rm(i)}] $\mathbf S_1$ and $\mathbf S_2$ are unitary,
\item[{\rm(ii)}] $\mathbf S_1$ is unitary and $\mathbf S_2$ is idempotent,
\item[{\rm(iii)}] $\mathbf S_1$ and $\mathbf S_2$ are idempotent,
\item[{\rm(iv)}] $\mathbf S_1$ and $\mathbf S_2$ are rings and $\BId(\mathbf S_1\times\mathbf S_2)$ is distributive.
\end{enumerate}
Then every ideal of $\mathbf S_1\times\mathbf S_2$ is directly decomposable.
\end{corollary}

\begin{proof}
Assume $(a,b)\in I\in\Id(\mathbf S_1\times\mathbf S_2)$. Then \\
$(a,0)=(a,b)(1,0)\in I$ and $(0,b)=(a,b)(0,1)\in I$ in case (i), \\
$(a,0)=(a,b)(1,0)\in I$ and $(0,b)=(a,b)(0,b)\in I$ in case (ii) and \\
$(a,0)=(a,b)(a,0)\in I$ and $(0,b)=(a,b)(0,b)\in I$ in case (iii) \\
showing direct decomposability of $I$ according to condition (iii) of Theorem~\ref{th1}. In case (iv), direct decomposability of $I$ follows from condition (ii) of Theorem~\ref{th1}.
\end{proof}

If a field $\mathbf F=(F,+,\cdot)$ is considered as a ring then it has no proper ideals. However, the same is valid also in the case of semiring ideals. Namely, if $I$ is a non-zero semiring ideal in $\mathbf F$ and $d\in I\setminus\{0\}$, then for each $x\in F$ we have $x=dd^{-1}x\in I$ proving $I=F$.

\begin{proposition}
If $\mathbf S$ is a commutative semiring and $\mathbf F$ a field then every ideal of $\mathbf S\times\mathbf F$ is directly decomposable.
\end{proposition}

\begin{proof}
Assume $\mathbf S=(S,+,\cdot,0)$ and $\mathbf F=(F,+,\cdot)$, let $I\in\Id(\mathbf S\times\mathbf F)$ and put $J:=\pi_1(I)$. If $I\subseteq S\times\{0\}$ then $I=J\times\{0\}$. Now assume $I\not\subseteq S\times\{0\}$. Then there exists some $(a,b)\in I$ with $b\neq0$. If $(c,d)\in J\times F$ then $c\in J=\pi_1(I)$. Thus there exists some $e\in F$ with $(c,e)\in I$ and hence
\[
(c,d)=(c,e)+(a,b)(0,b^{-1}(d-e))\in I
\]
showing $I=J\times F$. Hence, $\mathbf S\times\mathbf F$ has directly decomposable ideals.
\end{proof}

\section{Congruence kernels of direct products of commutative semirings}

Now we draw our attention to congruence kernels.

If $\Theta_1\in\Con\mathbf S_1$ and $\Theta_2\in\mathbf S_2$ then
\[
\Theta_1\times\Theta_2:=\{((x_1,x_2),(y_1,y_2))\mid(x_1,y_1)\in\Theta_1,(x_2,y_2)\in\Theta_2\}\in\Con(\mathbf S_1\times\mathbf S_2)
\]
and $[(0,0)](\Theta_1\times\Theta_2)=[0]\Theta_1\times[0]\Theta_2$. However, there may exist congruences $\Theta$ on $\mathbf S_1\times\mathbf S_2$ such that $[(0,0)]\Theta\neq[0]\Theta_1\times[0]\Theta_2$ for all possible $\Theta_1\in\Con\mathbf S_1$ and $\Theta_2\in\Con\mathbf S_2$. It should be noted that Fraser and Horn (cf.\ \cite{FH}) presented necessary and sufficient conditions for direct decomposability of congruences. In the following we will modify these conditions for congruence kernels.

If $\mathbf S_1=(S_1,+,\cdot,0)$ and $\mathbf S_2=(S_2,+,\cdot,0)$ are commutative semirings, $\Theta\in\Con(\mathbf S_1\times\mathbf S_2)$ and $i\in\{1,2\}$ then we put
\begin{align*}
\pi_i(\Theta) & :=\{(\pi_i(x),\pi_i(y))\mid(x,y)\in\Theta\}, \\
        \Pi_i & :=\{(x,y)\in(S_1\times S_2)^2\mid\pi_i(x)=\pi_i(y)\}.
\end{align*}
Note that $\Theta_i:=\pi_i(\Theta)\in\Con(\mathbf S_i)$, $\Pi_i\in\Con(\mathbf S_1\times\mathbf S_2)$ and $\Pi_1\cap\Pi_2=\{(x,x)\mid x\in S_1\times S_2\}$. Let us remark that in general $[0]\Theta_i\neq\pi_i([(0,0)]\Theta)$, namely e.g.\ $a\in[0]\Theta_1$ is equivalent to the fact that there exist $b,c\in S_2$ with $(a,b)\in[(0,c)]\Theta$.

We call the kernel $[(0,0)]\Theta$ {\em directly decomposable} if
\[
[(0,0)]\Theta=\pi_1([(0,0)]\Theta)\times\pi_2([(0,0)]\Theta),
\]
and furthermore, we call the kernel $[(0,0)]\Theta$ {\em strongly directly decomposable} if
\[
[(0,0)]\Theta=[0]\Theta_1\times[0]\Theta_2.
\]
Note that
\begin{align*}
\pi_i([(0,0)]\Theta) & \subseteq[0]\Theta_i\text{ for }i=1,2, \\
       [(0,0)]\Theta & \subseteq\pi_1([(0,0)]\Theta)\times\pi_2([(0,0)]\Theta),
\end{align*}
thus strongly direct decomposability implies direct decomposability (cf.\ also the following Theorems~\ref{th2} and \ref{th3}).

We characterize strongly directly decomposable congruence kernels as follows:

\begin{theorem}\label{th2}
If $\mathbf S_1=(S_1,+,\cdot,0)$ and $\mathbf S_2=(S_2,+,\cdot,0)$ are commutative semirings and $\Theta\in\Con(\mathbf S_1\times\mathbf S_2)$ then $[(0,0)]\Theta$ is strongly directly decomposable if and only if the following holds:
\[
\text{If }(a,b)\T(0,c)\text{ and }(d,e)\T(f,0)\text{ then }(a,e)\T(0,0)
\]
for $(a,b),(0,c),(d,e),(f,0)\in S_1\times S_2$.
\end{theorem}

\begin{proof}
If $[(0,0)]\Theta$ is strongly directly decomposable and
\[
(a,b)\T(0,c)\text{ and }(d,e)\T(f,0)
\]
for $(a,b),(0,c),(d,e),(f,0)\in S_1\times S_2$ then $(a,e)\in[0]\Theta_1\times[0]\Theta_2=[(0,0)]\Theta$. If, conversely, the condition of the theorem holds and $(g,h)\in[0]\Theta_1\times[0]\Theta_2$ then there exist $i,j\in S_1$ and $k,l\in S_2$ with $(g,k)\T(0,l)$ and $(i,h)\T(j,0)$ and hence $(g,h)\in[(0,0)]\Theta$ showing $[0]\Theta_1\times[0]\Theta_2\subseteq[(0,0)]\Theta$. The converse inclusion is trivial.
\end{proof}

Using this result we can prove the following

\begin{theorem}\label{th4}
If $\mathbf S_1$ and $\mathbf S_2$ are commutative semirings, $\Theta\in\Con(\mathbf S_1\times\mathbf S_2)$ and
\begin{align*}
[(0,0)]((\Theta\vee\Pi_1)\cap\Pi_2) & \subseteq[(0,0)]\Theta, \\
[(0,0)]((\Theta\vee\Pi_2)\cap\Pi_1) & \subseteq[(0,0)]\Theta
\end{align*}
then $[(0,0)]\Theta$ is strongly directly decomposable.
\end{theorem}

\begin{proof}
Let $(a,b),(0,c),(d,e),(f,0)\in S_1\times S_2$ and assume $(a,b)\T(0,c)$ and $(d,e)\T(f,0)$. Then
\[
(a,0)\PO(a,b)\T(0,c)\PO(0,0)\text{ and }(a,0)\PT(0,0)
\]
and hence
\[
(a,0)\in[(0,0)]((\Theta\vee\Pi_1)\cap\Pi_2)\subseteq[(0,0)]\Theta.
\]
Analogously,
\[
(0,e)\PT(d,e)\T(f,0)\PT(0,0)\text{ and }(0,e)\PO(0,0)
\]
and hence
\[
(0,e)\in[(0,0)]((\Theta\vee\Pi_2)\cap\Pi_1)\subseteq[(0,0)]\Theta.
\]
Therefore
\[
(a,e)=(a,0)+(0,e)\T(0,0)+(0,0)=(0,0).
\]
Now the assertion follows from Theorem~\ref{th2}.
\end{proof}

Recall that an {\em algebra} $\mathbf A$ with $0$ is called {\em distributive at $0$} (see e.g.\ \cite{CEL12}) if for all $\Theta,\Phi,\Psi\in\Con\mathbf A$ we have
\[
[0]((\Theta\vee\Phi)\cap\Psi)=[0]((\Theta\cap\Psi)\vee(\Phi\cap\Psi)).
\]
A {\em class} of algebras with $0$ is called {\em distributive at $0$} if any of its members has this property. Applying Theorem~\ref{th4} we can state:

\begin{corollary}
If $\mathbf S_1$ and $\mathbf S_2$ are commutative semirings and $\mathbf S_1\times\mathbf S_2$ is distributive at $(0,0)$ then the kernel of every congruence on $\mathbf S_1\times\mathbf S_2$ is strongly directly decomposable.
\end{corollary}

\begin{proof}
For all $\Theta\in\Con(\mathbf S_1\times\mathbf S_2)$ we have
\begin{align*}
[(0,0)]((\Theta\vee\Pi_1)\cap\Pi_2) & =[(0,0)]((\Theta\cap\Pi_2)\vee(\Pi_1\cap\Pi_2))=[(0,0)](\Theta\cap\Pi_2)\subseteq[(0,0)]\Theta, \\
[(0,0)]((\Theta\vee\Pi_2)\cap\Pi_1) & =[(0,0)]((\Theta\cap\Pi_1)\vee(\Pi_2\cap\Pi_1))=[(0,0)](\Theta\cap\Pi_1)\subseteq[(0,0)]\Theta.
\end{align*}
Now the result follows from Theorem~\ref{th4}.
\end{proof}

Recall the Mal'cev type characterization of distributivity at $0$ from \cite{CEL12} (Theorem~8.2.2):

\begin{proposition}\label{prop1}
A variety with $0$ is distributive at $0$ if and only if there exist some positive integer $n$ and binary terms $t_0,\ldots,t_n$ such that
\begin{align*}
t_0(x,y) & \approx0, \\
t_i(0,y) & \approx0\text{ for }i\in\{0,\ldots,n\}, \\
t_i(x,0) & \approx t_{i+1}(x,0)\text{ for even }i\in\{0,\ldots,n-1\}, \\
t_i(x,x) & \approx t_{i+1}(x,x)\text{ for odd }i\in\{0,\ldots,n-1\}, \\
t_n(x,y) & \approx x.
\end{align*}
\end{proposition}

We can apply Proposition~\ref{prop1} to idempotent commutative semirings.

\begin{corollary}
The variety of idempotent commutative semirings is distributive at $0$ and hence has strongly directly decomposable congruence kernels.
\end{corollary}

\begin{proof}
If
\begin{align*}
       n & :=2, \\
t_0(x,y) & :=0, \\
t_1(x,y) & :=xy, \\
t_2(x,y) & :=x
\end{align*}
then
\begin{align*}
t_0(x,y) & \approx0, \\
t_0(0,y) & \approx0, \\
t_1(0,y) & \approx0y\approx0, \\
t_2(0,y) & \approx0, \\
t_0(x,0) & \approx0\approx x0\approx t_1(x,0), \\
t_1(x,x) & \approx xx\approx x\approx t_2(x,x), \\
t_2(x,y) & \approx x
\end{align*}
and hence, by Proposition~\ref{prop1}, we obtain the result.
\end{proof}

The following characterization of directly decomposable congruence kernels is similar to the characterization presented in Theorem~\ref{th2}.

\begin{theorem}\label{th3}
For commutative semirings $\mathbf S_1=(S_1,+,\cdot,0)$ and $\mathbf S_2=(S_2,+,\cdot,0)$ and $\Theta\in\Con(\mathbf S_1\times\mathbf S_2)$ the kernel $[(0,0)]\Theta$ is directly decomposable if and only if
\begin{equation}\label{equ2}
(a,b),(c,d)\in[(0,0)]\Theta\text{ implies }(a,d)\in[(0,0)]\Theta.
\end{equation}
\end{theorem}

\begin{proof}
If the kernel $[(0,0)]\Theta$ is directly decomposable and $(a,b),(c,d)\in[(0,0)]\Theta$ then
\[
a\in\pi_1([(0,0)]\Theta)\text{ and }d\in\pi_2([(0,0)]\Theta)
\]
whence
\[
(a,d)\in\pi_1([(0,0)]\Theta)\times\pi_2([(0,0)]\Theta)=[(0,0)]\Theta
\]
proving (\ref{equ2}). Conversely, assume (\ref{equ2}) to be satisfied and let $(a,d)\in\pi_1([(0,0)]\Theta)\times\pi_2([(0,0)]\Theta)$. Then there exist $b\in S_2$ and $c\in S_1$ with $(a,b),(c,d)\in[(0,0)]\Theta$. Using (\ref{equ2}) we conclude $(a,d)\in[(0,0)]\Theta$ proving
\[
\pi_1([(0,0)]\Theta)\times\pi_2([(0,0)]\Theta)\subseteq[(0,0)]\Theta.
\]
The converse inclusion is trivial.
\end{proof}

It is evident also from the conditions of Theorems~\ref{th2} and \ref{th3} that if a direct product of semirings has strongly directly decomposable congruence kernels then it has directly decomposable congruence kernels.

We say that a class $\mathcal C$ of algebras of the same type containing a constant $0$ has {\em directly decomposable congruence kernels} if for any $\mathbf A_1,\mathbf A_2\in\mathcal C$ and each $\Theta\in\Con(\mathbf A_1\times\mathbf A_2)$, $[(0,0)]\Theta$ is directly decomposable.

The following Mal'cev type condition was derived in \cite{CEL12}:

\begin{proposition}\label{prop2}
{\rm(}Theorem 11.0.4 in {\rm\cite{CEL12})} A variety of algebras with $0$ has directly decomposable congruence kernels if there exist positive integers $m$ and $n$, binary terms $s_1,\ldots,s_m,t_1,\ldots$ $\ldots,t_m$ and $(m+2)$-ary terms $u_1,\ldots,u_n$ satisfying the identities
\begin{align*}
u_1(x,y,s_1(x,y),\ldots,s_m(x,y)) & \approx x, \\
u_1(y,x,t_1(x,y),\ldots,t_m(x,y)) & \approx x, \\
u_i(y,x,s_1(x,y),\ldots,s_m(x,y)) & \approx u_{i+1}(x,y,s_1(x,y),\ldots,s_m(x,y))\text{ for }i=1,\ldots,n-1, \\
u_i(x,y,t_1(x,y),\ldots,t_m(x,y)) & \approx u_{i+1}(y,x,t_1(x,y),\ldots,t_m(x,y))\text{ for }i=1,\ldots,n-1, \\
u_n(y,x,s_1(x,y),\ldots,s_m(x,y)) & \approx x, \\
u_n(x,y,t_1(x,y),\ldots,t_m(x,y)) & \approx y.
\end{align*}
\end{proposition}

\begin{corollary}
The variety of unitary commutative semirings has directly decomposable congruence kernels.
\end{corollary}

\begin{proof}
If
\begin{align*}
             m & :=3, \\
             n & :=2, \\
      s_1(x,y) & :=1, \\
      s_2(x,y) & :=0, \\
      s_3(x,y) & :=0, \\
      t_1(x,y) & :=0, \\
      t_2(x,y) & :=1, \\
      t_3(x,y) & :=y, \\
u_1(x,y,z,u,v) & :=xz+yu, \\
u_2(x,y,z,u,v) & :=yz+v
\end{align*}
then
\begin{align*}
u_1(x,y,1,0,0) & \approx x1+y0\approx x, \\
u_1(y,x,0,1,y) & \approx y0+x1\approx x, \\
u_1(y,x,1,0,0) & \approx y1+x0\approx y\approx y1+0\approx u_2(x,y,1,0,0), \\
u_1(x,y,0,1,y) & \approx x0+y1\approx y\approx x0+y\approx u_2(y,x,0,1,y), \\
u_2(y,x,1,0,0) & \approx x1+0\approx x, \\
u_2(x,y,0,1,y) & \approx y0+y\approx y
\end{align*}
and hence, by Proposition~\ref{prop2}, we obtain the result.
\end{proof}

Authors' addresses:

Ivan Chajda \\
Palack\'y University Olomouc \\
Faculty of Science \\
Department of Algebra and Geometry \\
17.\ listopadu 12 \\
771 46 Olomouc \\
Czech Republic \\
ivan.chajda@upol.cz

G\"unther Eigenthaler \\
TU Wien \\
Faculty of Mathematics and Geoinformation \\
Institute of Discrete Mathematics and Geometry \\
Wiedner Hauptstra\ss e 8-10 \\
1040 Vienna \\
Austria \\
guenther.eigenthaler@tuwien.ac.at

Helmut L\"anger \\
TU Wien \\
Faculty of Mathematics and Geoinformation \\
Institute of Discrete Mathematics and Geometry \\
Wiedner Hauptstra\ss e 8-10 \\
1040 Vienna \\
Austria, and \\
Palack\'y University Olomouc \\
Faculty of Science \\
Department of Algebra and Geometry \\
17.\ listopadu 12 \\
771 46 Olomouc \\
Czech Republic \\
helmut.laenger@tuwien.ac.at
\end{document}